\documentclass{amsbook}

\usepackage{amssymb,amsfonts, amsmath, amsthm}
\usepackage{euscript}

\newtheorem{theorem}[subsection]{Theorem}
\newtheorem{proposition}[subsection]{Proposition}
\newtheorem{corollary}[subsection]{Corollary}
\newtheorem{lemma}[subsection]{Lemma}

\newtheorem{definition}[subsection]{Definition}

\theoremstyle{remark}

\newtheorem{remark}[subsection]{Remark}    
\newtheorem{example}[subsection]{Example}
\newtheorem{nota}{Notation}

\usepackage{hyperref}

\usepackage{amsrefs}

\usepackage[all]{xy}
\xyoption{all}

\newcommand{\Hom}{\mathrm{Hom}}
\newcommand{\Cont}{\mathrm{Cont}}
\newcommand{\maps}{\mathrm{Maps}}

\newcommand{\CP}{\mathbb{C}P}
\newcommand{\Loc}{ L_{K(1)}}
\newcommand{\sm}{\wedge}

\newcommand{\tline}{\mathbb{L}}

\newcommand{\G}{\widehat{\mathbb{G}}}
\newcommand{\expf}{\exp_F}
\newcommand{\expg}{\exp_G}
\newcommand{\logg}{\log _G }
\newcommand{\logf}{\log _F }

\newcommand{\Z}{\mathbb{Z}}
\newcommand{\F}{\mathbb{F}}
\newcommand{\Q}{\mathbb{Q}}
\newcommand{\C}{\mathbb{C}}

\newcommand{\spf}{\mathrm{Spf}}
\newcommand{\spec}{\mathrm{Spec}}
\newcommand{\affline}{  \widehat{\mathbb{A}}_1}
\newcommand{\strg}{O_G}

\begin{document}

\title{Orientations and $p$-adic analysis}

\author{Barry John Walker}             
\email{bjwalker@math.northwestern.edu}       
\address{Mathematics Department,
         Northwestern University,
         2217 Sheridan Avenue,
         Evanston, IL 60202,
         USA}



\begin{abstract}
Matthew Ando produced power operations in the Lubin-Tate cohomology theories and was 
able to classify
which complex orientations  were compatible with these  operations.
The methods used by Ando, Hopkins and Rezk to classify orientations of topological modular forms
can be applied to complex K-Theory.
Using techniques from local analytic number theory, we construct a theory of
integration on formal groups of finite height. 
This calculational device allows us to  show the equivalence of the
two descriptions for complex K-Theory.  As an application we show that the $p$-completion of 
the Todd genus is an $E_\infty$ map.
\end{abstract}


\maketitle

\tableofcontents




\section{Introduction}

The theory of power operations in cohomology theories arose with Steenrod operations
in Eilenberg MacLane spectra over $\F_p$.
Another example is Atiyah's construction of  power operations in complex K-Theroy using representation theory techniques \cite{MR0202130}.
In both
cases,a novel description of cohomology operations was obtained.
Other examples include Tomaoa tom Dieck power operations for bordism theories associated to classical Lie groups
\cite{MR0244989}.

In general, power operations  are internal properties of a spectrum. 
Thus a natural question
to ask is when such structure is preserved by maps of spectra.  Recall that an orientation is a map of spectra
whose source is the Thom spectrum of a space of $BO$.
Recently, Ando constructed  power operations for the Lubin-Tate spectra $E_n$ \cite{MR1344767}.  
In addition, he
produced a necessary and sufficient condition for a complex orientation of $E_n$ to be compatible
with the power operations in complex cobordism and $E_n$.



Continuing this
program, Ando along with Hopkins and Strickland generalized this result to different types of Thom spectra associated to
various spaces over $BO$ 
\cite{MR2045503}.
One of the goals of their article was to explain the existence of the Witten genus in a purely homotopical way.  
In \cite{AHR} an $E_\infty$ map from the Thom spectrum of the eight connected cover of
 $BO$,  $BO \langle 8 \rangle $,
to the spectrum of topological modular forms was produced.  The projection of this map to the elliptic spectrum
attached to the Tate curve produces the Witten genus.
In order to produce the map to topological modular forms, they needed to study congruences 
among the Bernoulli numbers of a formal group.
An outline of this program, along with its overall place in the tapestry that is mathematics, can be found in
the inspiring article of Hopkins \cite{MR1403956}.


In this note we  describe how the necessary and sufficient condition of Ando implies  the $E_\infty$ condition 
for an orientation of $p$-adic complex K-Theory.  It will turn out that these conditions are equivalent.  
One of the key ingredients is a theorem of Nicholas Katz involving measures on a formal 
group.  He was interested in producing congruences for sequences attached to elliptic
curves over local number fields that arise as the values of $p$-adic L-series
at negative integers.  We use isogonies on formal groups to produce analogous
sequences and congruences between them.
We believe that
this calculation locates the role of isogenies in the description of \cite{AHR}.
I would like to thank both Matthew Ando and Charles Rezk for their infinite patience.

\subsection{Notation}

Recall that a formal group law over a complete local noetherian ring $R$ is a power series $f(x, y) \in R[[x,y]]$ that satisfies
the formal properties of a group law.  If $F(x,y)$ is a such a power series we will often write
\[
x+ _F y
\]
for the formal sum.

For any $a \in \Z$ we write $[a]_F (t)$ for the
power series obtained by adding $t$ to itself $a$ times over the formal group law.  If
$a$ is negative, then we use the inverse power series.  See the appendix of \cite{MR860042}.
In the case when $F$ is a Lubin-Tate formal group law over $R$
we  write $[a]_F (t)$ for the endomorphism of $F$ induced by $a \in R$.
For a formal group law we write $F[p^n]$ for the set of $p^n$-torsion points of $F$.  This set is computed by 
factoring the $p^n$ series using the Weierstrass preparation theorem.

An (affine commutative dimension one) formal group $G$ is a functor from complete local rings 
to abelian groups that is isomorphic
to the formal affine line $\affline$  after forgeting to sets.
Examples include the formal multiplicative group $\G_m = \spf \Z_p[u, u^{-1}]$ and the 
formal additive group  $\G_a = \spf[\Z_p[u]$.
The algebraic
multiplicative group will be written as $\mathbb{G}_m = \spec \Z_p[u, u^{-1}]$. 
All formal groups $G$ are over the complete local ring $R$ unless otherwise stated.

A coordinate on a formal group $G$ is a choice of isomorphism of functors between $G$ and $\affline$.
Such an isomorphism produces a formal group law by considering the sum of the two natural maps 
$t_1, t_2 \colon R[[t]] \to R[[t_1, t_2]]$.

For a prime $p$, $\Z_p$ is the $p$-adic integers and $\Q_p$ the $p$-adic rationals.
The spectrum $K_p$ is the $p$-completion of periodic complex K-Theory.  It is
a $K(1)$ local $E_\infty$ ring spectrum.
We write $MU$ for the Thom spectrum 
associated to the space $BU$  over $BO$.

\subsection{Structure}

This article has three main parts.  In the first we discuss the role of measures in algebraic topology.  We
build up the language of integration on a formal group  and how to calculate in that setting.  In the middle
of the article we define Bernoulli numbers analogously to modular forms.  We then  investigate some of their properties
and how they arise in $K(1)$ local homotopy theory.
Finally, we give a brief review of the obstruction theory for $E_\infty$ orientations 
and verify the conditions we need for our story.
We close with a proof of the main result.

\section{The Main Result}

Suppose $\alpha \colon MU \to K_p$ is a homotopy ring map.  By Quillen's theorem $\alpha$ determines a formal group
law $F$ over the ring  $\Z_p$.  Let $\expf x$ be the strict isomorphism from the formal additive group law  to  $F$ 
over $\Q_p$. 
Expanding  the Hirzebruch series as 
\[
\frac{x}{\expf x} = \exp \left( \sum\limits_{k \geq 1 }  t_k \frac{x^k}{k!} \right)
\]
produces a sequence $\left\{ t_k \right\}$ attached to the orientation $\alpha$.

\begin{theorem} \label{thm1}
For an odd prime $p$ the map 
\[
\mathbb{D} : H_\infty (MU, K_p) \to \pi_0 E_\infty (MU, K_p).
\]
that sends $\alpha$ to the sequence $\left\{ t_k \right\}$ 
is an isomorphism.  Moreover, $\mathbb{D}$ is a section of the forgetful map.
\end{theorem}

Suppose $E$ is an  $E_\infty$ ring spectrum and $X$ is a space. 
Part of the data attached to $E$ is a collection of
functions
\[
E^0 ( E \Sigma_n \times_{\Sigma_n} X^n ) \xrightarrow{P} E^0 (X)
\]
for each  positive integer $n$  satisfying some formal identities called power operations.
More precise definitions along with some applications can be found in \cite{MR836132}.
Fix $E_\infty$ spectra  $E$ and $F$.  The set of homotopy classe
\[
\xymatrix{
E^0 (X) \ar[d]^f  \ar[r]^-{P_E} &  E^0 ( E \Sigma_n \times_{\Sigma_n} X^n )       \ar[d]_f      \\ 
F^0 (X)  \ar[r]_-{P_F}  &   F^0 ( E \Sigma_n \times_{\Sigma_n} X^n )    
}
\]
commutes for all spaces $X$ is written $H_\infty (E, F)$.  
The space of of $E_\infty$ maps from $E$ to $F$ is denoted  $E_\infty (E, F)$.


The method for classifying the components 
of the space of $E_\infty$ maps from complex
cobordism to $K_p$ is described in  \cite{AHR}.   Each component corresponds to a sequence $\left\{ t_k \right\}$ satisfying 
some conditions.  We will make these conditions explicit in Theorem \ref{ahrtheorem}.   
The source of $\mathbb{D}$ was completely determined by
Ando \cite{MR1344767}.
We will show that the necessary and sufficient condition 
describing the source 
is enough to produce the required conditions on the sequence $\left\{ t_k \right\} $.



\subsection{Background}

Quillen calculated the effect of the power operations in complex cobordism on the tautological line bundle
over $\CP^\infty$ \cite{MR0290382}.
One of Ando's insights was the similarity between this calculation of  Quillen
and  an isogeny  in the theory of formal groups, which we describe.  
Given a finite subgroup of a formal group 
Lubin constructs the quotient  formal group \cite{MR0209287}.
Part of the data is an isogeny from the group to the quotient.  This isogeny has the same general form
of Quillen's calculation.
Ando's work lead to the viewpoint that isogenies can be used to construct power operations in complex oriented cohomology
theories.

Suppose $p$ is a prime.  Let $E_1$ be the Lubin-Tate spectrum associated to a universal deformation of the Honda formal
group law over $\F_p$ of height one.
It is a $K(1)$ local $E_\infty$ ring spectrum and so has a theory of power operations.
For odd primes, $K_p$ is a model for $E_1$.

Since complex cobordism is the universal complex oriented spectrum, a map  
\begin{eqnarray} \label{mapalpha}
\alpha:  MU \to E_1
\end{eqnarray}
of naive homotopy ring spectra (compatible with smash products)
determines a graded formal group law over $\pi_* E_1$.  Since
$E_1$ is an even two periodic theory, we can consider the ungraded group law over
\[
\pi_0 E_1 \cong \Z_p.
\]
Call this formal group law $F_\alpha$.  We may drop the subscript $\alpha$ when it becomes notationally convenient.
The following result is the necessary and sufficient condition for $\alpha$ to be an $H_\infty$ orientation.
\begin{theorem}[Ando] \label{andocondition} 
The orientation $\alpha$ (\ref{mapalpha})
commutes with the power operations in the source and target if and only if 
\[
\prod\limits_{c \in F[p] } t +_F c = [p]_F (t).
\]
\end{theorem}

It is convenient to use the language of measures to describe the target of $\mathbb{D}$.
A $\Z_p$ valued measure on $\Z_p$ is a (continuous) linear functional 
\[
\mu : \Cont (\Z_p, \Z_p) \to \Z_p.
\]
We denote the collection of all such by $M(\Z_p, \Z_p)$.
The moments of $\mu$ is the sequence obtained by applying $\mu$ to the function $x^k$ for $k \geq 0$.
A description of the target in Theorem \ref{thm1}  is
\begin{theorem}[Ando, Hopkins, Rezk] \label{ahrtheorem}
For odd primes $p$, there is a bijection between elements $\alpha \in \pi_0 E_\infty (MU, K_p)$
and sequences 
\[
\left\{ t_k \right\} \in  \prod\limits_{k \geq 1} \Q_p
\]
satisfying
\begin{enumerate}
 \item For any $a \in \Z_p ^\times - \left\{ \pm 1 \right\} $ there exists a measure $\mu$ on the $p$-adic units, 
depending on $a$, such that
\[
\mu (x^k) = \int\limits_{\Z_p^\times} x^k d \mu (x) =  t_k( 1 -a^k)(1 - p^{k-1} ) 
\]
for all $k \geq 1$ and 
\item $ - \frac{B_k}{k} \equiv t_k \mod \Z_p$.

\end{enumerate}

\end{theorem}

The rational number  $B_k$ is the classical $k$-th Bernoulli number with generating series
\[
\frac{x}{e^x -1 } = \sum\limits_{k \geq 0 } B_k \frac{x^k}{k!}.
\]
The congruence in condition two means that the difference between the $p$-adic rationals $t_k$ and $- B_k / k$ 
is a $p$-adic integer.


\section{Cohomology Operations}

In this section, we will recall the connection between $p$-adic integration and operations in $K_p$.
The relationship between Bernoulli numbers and $p$-adic integration has an interesting relationship
with topological K-Theory.  One of the first authors to study this phenomenon was Francis Clarke \cite{pclarke}.
Additional work  involving the Kummer congruences and its role in the homotopy of the space $BU$
can be found in \cite{MR942424}.

Let $K(1)$ be Morava K-Theory of height one at the odd prime $p$ and
 $\Loc$ Bousfield localization with respect to the homology theory $K(1)$.
Suppose $f \colon S \to \Loc  \left( K_p \sm K_p \right) $, i.e. $f \in \left(K_p \right)_0^\wedge K_p$
and $a \in \Z_p^\times$.  If $\psi^a$ is the Adams operation at $a$, 
then the  composition
\[
S \xrightarrow{f} \Loc \left( K_p \sm K_p \right)  \xrightarrow{\psi^a \sm id} \Loc  \left( K_p \sm K_p \right)  \to \Loc K_p \cong K_p.
\]
is an element of $\pi_0 K_p \cong \Z_p$.  Letting $a$ vary shows that  elements of
$\pi_0 \Loc \left(  K_p \sm K_p  \right) $ are continuous
functions on $\Z_p^\times$.  Dualizing this observations yields

\begin{theorem}[Adams, Harris, Switzer] \label{kpmeasure}
There is an isomorphism
\[
\Hom_{\Cont} (\Cont (\Z_p^\times, \Z_p), \Z_p) = M(  \Z_p^\times , \Z_p)  \cong [K_p, K_p] = K_p^0 K_p.
\]
\end{theorem}

\begin{proof}
This statement follows from the three authors' work on operations in integral K-Theory \cite{MR0293617}.  
They describe the
operations in terms of the generalized binomial coefficients $ x \choose i $.  
Using $K(1)$ localization techniques, we  can produce this statement from the integral result.  
Some details are  worked out in \cite{AHR}
and in \cite{knlocal}.
A result of this type for the Lubin-Tate  spectra $E_n$ for any $n \geq 1$ can be found in   \cite{MR2076002}.
\end{proof}
It is shown in \cite{AHR} that the effect in $\pi_{2k}$ is given by multiplication by the $k$-th moment, $\mu(x^k)$.
Later in the article (Theorem \ref{effectcp}) we will describe the effect of an operation on the K-Theory
of infinite complex projective space $\CP^\infty$.










\section{Measure Theory}

In order to construct the map in Theorem \ref{thm1} we will use the theory of integration
on a formal group of finite height.

\subsection{$p$-divisible groups}

A $p$-divisible group of height $h$ is a system 
\[
(G_n, i_n)
\]
for $n \geq 1$ where
\begin{enumerate}

 \item Each $G_n$ is a locally  free commutative group scheme of rank $p^{nh}$.

 \item $i_n : G_n  \to G_{n +1} $ is a homomorphism of commutative group schemes.

 \item The sequence
\[
0 \to G_n  \xrightarrow{i_n} G_{n +1 } \xrightarrow{p^n} G_{n + 1}
\] 
is exact.
\end{enumerate}

Given a $p$-divisible group, $(G_n, i_n)$ the colimit
\[
\lim\limits_{\rightarrow} G_n
\]
is a functor from rings to abelian groups.
This construction does not always yield a formal group.

Suppose  $G$ is a (affine commutative dimension one) formal group of finite height $h$ over  a complete local ring $R$.  
Suppose further
that the residue field is perfect and has characteristic $p$.
These hypothesis imply that multiplication by $p$ on $G$ is an isogeny, that is a finite free map of group schemes. 
If $\Gamma_n$ denotes the kernel of the isogeny $p^n$ on $G$  and $i_n : \Gamma_n \to \Gamma_{n + 1}$ is the induced
map on the kernels,
the system
\[
(\Gamma_n,  i_n)
\]
is a $p$-divisible group of height $h$ over $R$.
The current  hypothesis on $R$ forces the colimit of $(\Gamma_n, i_n)$ 
to be the formal group $G$.
See \cite{MR1610452} for more details on $p$-divisible groups and a proof of this fact.


The Cartier dual of $\Gamma_n$ is represented by the functor
\[
\underline{\Hom } _{\mathrm{Groups}} (\Gamma_n,  \mathbb{G}_m)
\]
from complete local rings  to abelian groups and is denoted by $\Gamma_n^\vee$.
The homomorphism $p: \Gamma_{n+1} \to \Gamma_n $ induces a homomorphism $i^\vee_n : \Gamma^\vee_n \to \Gamma^\vee_{n + 1}$ 
on the duals and 
the data
\[
(\Gamma_n^\vee, i^\vee _n )
\]
is a $p$-divisible group.
As before,  the colimit of $(\Gamma'_n, i'_n)$ is a formal group
of height equal to the height of $G$.

\begin{nota}
The $p$-divisible dual of the formal group $G$ is denoted by   $G^\vee$.
\end{nota}

\subsection{Tate Modules}

Suppose $G$ is a formal group.  Recall that $G[p^n]$ is the set of $p^n$ torsion points of $G$.
\begin{definition}
The Tate module associated to a formal group $G$ is 
\[
T_p G := \lim\limits_{\leftarrow} G[p^n].
\] 
\end{definition}
The Tate module  is a $\Z_p$ module of rank equal to the height of $G$.
The Tate module of the $p$-divisible dual of $G$ is 
the inverse limit of its $p^i$ torsion points of the dual:
\[
T_p G^\vee = \lim\limits_{\leftarrow} G^\vee[p^i].
\]

\begin{lemma}[Tate] \label{tate_thm}
If $R_C$ is the ring of integers in the completion of an algebraic closure of $R$ then 
there is an isomorphism of $\Z_p$ modules
\[
T_p G^\vee \to \Hom_{\mathrm{R_C-Groups}} ( G \times {R_C}, {\mathbb{G}_m} \times {R_C} ).
\]
\end{lemma}
It is well known that when $G$ is a Lubin-Tate group of height one we can replace $R_C$ with the ground ring $R$.

\subsection{Measures on $G$}

For now on, we insist that the height of $G$ is one.
In this case, the Tate module of the dual of $G$,  $T_p G^\vee$
is isomorphic to $\Z_p$ as a $\Z_p$-module.
The 
Tate module possesses  a topology induced by the inverse limit and thus
we can consider the $R$ module 
\[
\Cont (T_p G^\vee , R)
\]
of continuous functions to $R$.

\begin{definition}[\cite{MR0441928}]
IF $S$ is a complete local algebra over $R$ then we define
the $S$-valued measures on $G$ as  
\[
M(G, S) = \Hom_R^\Cont ( \Cont ( T_p G^\vee, R), S)
\]
\end{definition}

\begin{nota}
If $\mu \in M(G,S)$ and $f : T_p G^\vee \to R$ we write
\[
\int\limits_{G} f (x) d\mu (x)
\]
for the effect of $\mu$ on the function $f$.
\end{nota}

If $l \in \Z_p$ we can consider $l \in T_p G^\vee$ using the isomorphism of Tate (Theorem \ref{tate_thm}).
\begin{example}[Dirac Measure] If $l \in \Z_p$ define a measure $\psi_l \in M(G, S)$ by
\[
\int\limits_{G} f(x) d \psi_l (x) = f(l)
\]
\end{example}
The measure $d \psi _l$ corresponds to the Adams operation $\psi^l$ under the identification in Theorem \ref{kpmeasure}.
In light of  Theorem \ref{ahrtheorem} we are interested in integration on the units.
\begin{definition}
Let
\[
\left( T_p G^\vee \right)^\times
\]
be the complement of $p T_p G^\vee$ inside $T_p G^\vee$. 
\end{definition}
The units of the Tate module is isomorphic to $\Z_p^\times$ since  $G$ has height one.
It inherits a topology from $T_p G^\vee$
and as above write
\[
M(G^\times, S) 
\]
for the $S$-valued measures on $\left(T_p G^\vee \right)^\times$.

We will be interested in the image of the natural map 
\[
\mathrm{res} : M(G^\times, S) \to M(G, S)
\]
defined by restricting the source of a  continuous function and then integrating.  

\begin{definition}[N. Katz]
A measure $\mu$ is supported on the units if it is in the image of the restriction map. 
\end{definition}

\section{Derivations and Measures}
In this section we develop the tools to calculate the moments of a measure on a formal group.

\subsection{Mahler's Theorem}
Suppose $S$ is a complete local torsion free $\Z_p$ algebra.
Write $M(\Z_p, S)$ for the module of
linear functionals $\mu : \Cont(\Z_p, \Z_p) \to S$.
\begin{theorem}[Mahler] \label{mahlerthm}
There is an isomorphism
\[
M(\Z_p, S) \to S[[t]]
\]
\end{theorem}

\begin{proof}
The generalized binomial coefficients  
\[
 {x \choose n}   = \frac{  x (x -1) (x-2) ... (x - n + 1) }{n!}.
\]
are continuous functions from $\Z_p$ to itself.
Mahler  showed
that these functions generate all continuous functions from $\Z_p$ to itself \cite{MR0095821}.
In particular, if $f : \Z_p \to \Z_p$ is continuous then
\[
f (x) = \sum\limits_{n} a_n(f)  {x \choose n}  
\]
and the $a_n (f)$ tend to zero $p$-adically.  
To finish the proof, identify the integral of the function $ x \choose n $
with the $n$-th coefficient of the associated power series.
\end{proof}

Mahler's theorem does not easily produce the moments, $\mu(x^k)$ of a measure.
Work of Katz produces a computational method using derivations on a
formal group $G$.  We will review this work and produce a slight generalization 
of Mahler's Theorem.

\subsection{Derivations}

Suppose that $G$ is a formal group over $R$ with multiplication give by $m : G \times G \to G$. 
and ring of functions $\strg$.
Recall that $\strg \rtimes \strg$ is a commutative ring with multiplication given by
\[
(a, m) (b, n) = (ab, an + bm )
\]

\begin{definition}
A derivation on $G$ is a map of commutative rings
\[
\omega' : \strg \to \strg \rtimes \strg
\]
such that the diagram
\[
\xymatrix{
\strg \ar[d]_= \ar[r]^-{\omega'} &   \strg \rtimes \strg \ar[dl]^{p_1} \\
\strg 
}
\]
commutes.
\end{definition}
We write 
\[
\omega :  \strg \xrightarrow{\omega'}  \strg \rtimes \strg \xrightarrow{p_2} \strg
\]
for the compostion of $\omega'$ with projection onto the second factor.

\begin{definition}A derivation $\omega$ is invariant if the diagram of commutative rings
\[
\xymatrix{
\strg \ar[d]_{m^*}  \ar[r]^\omega   &  \strg   \ar[d]^{m^*} \\ 
\strg \otimes \strg \ar[r]^{\omega \otimes 1} &   \strg  \otimes \strg 
}
\]
commutes.
\end{definition}




\begin{example} Suppose $t$ is a coordinate on $G$ and $F(x,y)$ is the induced formal group.  
Write $F_2(x,y)$ for the partial derivative of $F(x,y)$ with respect to the second variable.
The map
\[
D_F = F_2 (t, 0) \frac{d}{dt} : \strg \to \strg
\]
is an invariant derivation.
\end{example}

\begin{definition} \label{betak}
If $\omega$ is an invariant derivation on $G$ define
\[
\beta^k : \strg \xrightarrow{\omega^k } \strg  \xrightarrow{0^*} R.
\]
where the exponent $k$ means compose $\omega$ with itself $k$ times.
\end{definition}

\subsection{Computing the Moments of a Measure}

If $t$ is a coordinate on $G$, elements of the Tate module of the $p$-divisible dual, 
$T_p G^\vee$,  are power series $g(t)$ over $R$ such that
$g ( x +_F y ) = g(x) g(y)$ and $g(0) = 1$. 
Given an invariant derivation  $\omega$ we can construct a  continuous function
\[
\langle \omega , - \rangle :  T_p G^\vee \to R  
\]
defined by 
\[
\langle \omega, g(t) \rangle = \omega g (t) |_{g = 0}   
\]
If $\mathrm{Diff}(G)$ is the $p$-completed module of invariant differential operators on $G$ then 
this  pairing  shows that $\strg$ is the continuous $R$-linear dual of  $\mathrm{Diff}(G)$ \cite{MR0441928}.

Let $S$ be a complet local torsion free algebra over $R$.
Hitting the map 
\[
\langle -, - \rangle \colon \mathrm{Diff}(G) \to \Cont( T_p G^\vee, R)
\]
with $\Hom_R^\Cont( -, S)$
produces a homomorphism
\[
\mathcal{D}_G \colon M(G, S) \to \strg \hat{\otimes} S 
\]

\begin{proposition} \label{diffop}
If $G$ is a height one Lubin-Tate formal group over $\Z_p$ with coordinate $t$, then  the natural map
\[
\mathcal{D}_G \colon M(G, S) \to \strg \hat{\otimes} S \cong S[[t]]  
\]
is an isomorphism.
\end{proposition}

\begin{proof}
If $t$ is a coordinate on $G$ then $a \in \Z_p$ can be identified with $F_2 (t, 0)^a \in T_p G^\vee$.
Thus the continuous function 
\[
 \langle D_F ^k , - \rangle  \colon T_p G^\vee \to \Z_p
\]
corresponds to the function
\[
x^k \colon \Z_p \to \Z_p.
\]
The result now follows from Theorem \ref{mahlerthm}.
\end{proof}

\begin{corollary} \label{mahlerderv}
Under the correspondence 
\[
\mu \mapsto f_\mu 
\]
of Proposition \ref{diffop}
we have the formula 
\[
\int\limits_{G} x^k  d \mu (x) = \beta^k f_\mu
\]
where the derivation we are using to define $\beta$ is $D_F$.
\end{corollary}

\begin{proof}
We can calculate.
Define $c_{m , k} $ by
\[
x^k = \sum\limits_{m = 0}^k c_{m, k } {x \choose m} 
\]
If $s$ is the coordinate on $\G_m = \spf \Z_p[u, u^{-1}]$ given by $u-1 $ then $D = (1 + s) \frac{d}{ds}$.  
It is easy to show that
\[
D^k x^m = c_{m , k}
\]
for $k \leq m$.
We can obtain the result for any height one Lubin-Tate group by pulling back the derivation $D$ along
an isomorphism to $\G_m$.  Explicitly, if $\Theta : G \to \G_m$ then pulling back induces
the formula
\[
\beta^k_F f (t) = \beta^k_{\G_m} f( \Theta(s)).
\]
\end{proof}

\begin{remark}
If $G$ produces a  Landweber exact theory $E$ with $\pi_0 E = R$, this statement can  be realized in topology by writing down the paring between
$E^0 \CP^\infty $ and $E^\vee_0 \CP^\infty$.  The later is the ring of invariant derivations on $G$ 
generated by $D_F^k$.
\end{remark}



\begin{definition}
If $\mu$ is a $S$-valued measure on $G$ then the sequence $ \beta^k ( f_\mu )$ for $k \geq 0$ is the moments of the measure $\mu$.
\end{definition}

The calculation  in \ref{mahlerderv} depends on a coordinate, however the isomorphism is natural in the following sense.
\begin{lemma} \label{d_compatible}
Suppose $s$ and $t$ are coordinates on the formal group $G$.  Then the diagram
\[
\xymatrix{
M(G, S) \ar[r] \ar[d]^=  &   S[[s]]  \ar@{-->}[d]^\Theta \\
M(G, S)  \ar[r] & S[[t]]
}
\]
commutes.
\end{lemma}

\begin{proof}
The coordinates $s$ and $t$ determine a pair of isomorphic formal group laws over $R$.  The associated strict isomorphism,
say $\Theta$, is the right hand vertical map in the above diagram.
\end{proof}

\section{A Theorem of Katz}

We are interested in the image of the restriction map
\[
\mathrm{res} \colon M(G^\times, S) \to M(G, S).
\]

\begin{theorem}[N. Katz] \label{katzthm}
Let $R$ be a complete local noetherian domain with residue characteristic $p$.
Suppose $G$ is a formal group of height \underline{one} over $R$ and $t$ is a coordinate on $G$.
Let $S$ be a complete  local $R$-algebra.
An $S$-valued  measure $\mu \in M( G , S)$ is supported on the units if and only if
$f_\mu \in S[[t]]$
satisfies
\[
Trace_G f = T_G f := \sum\limits_{c \in F[p] }  f_\mu (t +_F c)  = 0
\]
\end{theorem}

This statement, without proof, can be found in  \cite{MR0441928} where it is also claimed to be true
when $G$ is the formal group associated to a supersingular elliptic curve.
We will verify only the part of the  theorem we need for our purposes.  Namely when $G$ is a height one Lubin-Tate
group over $\Z_p$.  As a warm up, we check the result for $G = \G_m$ with a specified coordinate.  We  obtain
the general statement using deformation theory.


In Proposition \ref{effectcp} we show
that the effect of a cohomology operation on the tautological line bundle $\tline$ over $\CP^\infty$ is given by $f_\mu$.
This calculation allows us to show that Theorem \ref{katzthm} is equivalent to the Madsen, Snaith and Tornehave 
result describing the image of
\[
\Omega^\infty : [K_p, K_p] \to [\Omega^\infty K_p, \Omega^\infty K_p].
\] 
In other words,

\begin{corollary} 
There is a commutative diagram
\[
\xymatrix{
M( \G ^\times_m , \Z_p) \ar[r]^{\mathrm{res}}  \ar[d] &   M(\G_m, \Z_p) \ar[d] \\
[K_p, K_p] \ar[r]_-{\Omega^\infty}   &  [\Omega^\infty  K_p, \Omega^\infty  K_p].
}
\]
\end{corollary}


In order to prove Theorem \ref{katzthm} we will describe a measure on $G$ in terms of Riemann sums.  

\subsection{Measures on $\Gamma_n$ }


Recall that $\Gamma_n$ is the kenrel
of multiplication by $p$ on $G$.
Define $M(\Gamma_n, S)$ as  the $S$ module of linear functionals
\[
\Hom ( \maps ( T_p \Gamma_n^\vee , R) , S) = \Hom ( \maps ( \Gamma_n^\vee, R), S).
\]
If $\mu \in M(\Gamma_n, S)$ we will often write 
\[
\int\limits_{\Gamma_n } f (x) d \mu(x)
\]
for the effect of $\mu$ on the function $f$.

A coordinate $t$ on $G$ determines an isomorphism
\[
M(\Gamma_n, S) \cong \Hom ( \maps (\Z/p^n, R ), S). 
\]
and there is an isomorphism 
\[
M ( \Gamma_n, S ) \cong S [ \Z /p^k ] \cong S [\gamma] / \gamma^{p^n} - 1.
\]
If $\chi_i$ is the characteristic function at $i$ then 
\[
\mu \mapsto \sum\limits_{i = o}^{p^n -1 }  \left( \int\limits_{\Gamma_n } \chi_i (x) d \mu (x) \right) \gamma^i 
\]

Hitting the inverse system
\[
... \to   \Gamma_n \to \Gamma_{n - 1} \to ...
\]
with the functor $M( - , S)$  produces  the inverse system
\[
... \to M( \Gamma_n , S) \to M( \Gamma_{n - 1} , S) \to ....
\]
Each map  sends a measure $\mu$ on $\Gamma_n$  to the measure $\bar{\mu}$ defined by
\[
\int\limits_{\Gamma_{n - 1}}  f d \bar{\mu} := \int\limits_{\Gamma_n} \bar{f} d \mu. 
\]
Here $\bar{f}$ is the composite of  $f$ with the natural projection:
\[
\xymatrix{
\Z / p^{n-1}   \ar[r]^f &   \Z_p \\
\Z/p^n \ar[u] \ar[ru]_{\bar{f} }.  & 
}
\]

For example, if $a \in \Z/ p^n$ the  Dirac measure $d \psi_a $ is mapped to $d \psi_{\bar{a}}$ where $\bar{a}$ is the 
class represented by the reduction of $a$ modulo $p^{n-1}$.
In general we have 
\[
\sum\limits_{i = 0 }^ {p^n - 1} a_i (n) d \psi_i \mapsto \sum\limits_{i = 0 }^ {p^n - 1} a_i (n) d \psi_{\bar{i}}.
\]
However, each class modulo $p^{n-1}$ gets hit precisely $p$ times.  Collecting terms, we can rewrite this sum as
\[
\sum\limits_{j = 0 }^ {p^{n - 1} - 1}           \sum\limits_{l = 0}^{p-1}      a_{i + l p^{n-1}} (n)               d \psi_j.
\]
In other words, we have the formula
\begin{eqnarray} \label{aik-formula}
a_i ( n - 1 ) = \sum\limits_{l = 0}^{p-1} a_{i + l p^{n-1} } ( n).
\end{eqnarray}

\subsection{Density of Dirac Measures}


Recall from our discussion on Tate modules that 
$G$ is isomorphic to its associated $p$-divisible group.  
Putting this together with the discussion of the last section we have
\begin{lemma}
There is an isomorphism
\begin{eqnarray} \label{mudecomp}
M(G, S) \cong \lim\limits_{\leftarrow} M ( \Gamma_n , S).
\end{eqnarray}
\end{lemma}

\begin{proof} Observe 
\begin{eqnarray*}
M(   G,   S)    =   M  (\lim\limits_{\rightarrow} \Gamma_n , S) = \lim\limits_{\leftarrow} M (\Gamma_n , S).
\end{eqnarray*}
\end{proof}





We are now in a position to describe an analytic version of integration on $G$.
Suppose  $\mu$  is in $M(G, S)$.  The previous lemma permits us to write  $\mu$ 
uniquely as a sequence
\[
\left\{ s_n = \sum\limits_{i=0}^{p^n - 1} a_i(k) d \psi_i \right\}
\]
using the isomorphism in (\ref{mudecomp}).
A statement of the following  form can be found in \cite{MR0466081}.

\begin{proposition} \label{intform} 
With the above  notation, if  $f \colon T_pG^\vee \to R$ is any continuous function
then
\[
\int\limits_{G} f(x) d\mu(x) =
\lim\limits_{n \rightarrow \infty} \sum\limits_{i = 0 }^{p^n - 1} a_i(n) f( [\gamma]^i)
\]
\end{proposition}
\begin{proof}

We begin by showing that the limit converges. 
To see this recall that the $a_i(n)$ and $a_j(n + 1)$ are
related; we have the formula
\[
a_i(n) = \sum\limits_{l = 0}^{p-1} a_{i + lp^n} (n + 1)
\]
from line (\ref{aik-formula}).

If we rewrite the $n$-th term in the limit using this formula we obtain
\begin{eqnarray} \label{kthterm}
\sum\limits_{i = 0}^{p^n - 1} \sum\limits_{l = 0}^{p-1}
a_{i + lp^n} (n + 1) f(i)
\end{eqnarray}
We need to compare this sum with
\begin{eqnarray} \label{kplus1term}
\sum\limits_{j = 0}^{p^{n+1} - 1} a_j(n + 1) f(j).
\end{eqnarray}
We can rewrite line (\ref{kplus1term}) as
\[
\sum\limits_{i = 0 }^{ p^n - 1 } \sum\limits_{l = 0}^{p-1}
a_{ i + l p^n } (n + 1) f(i + l p^n).
\]

There are the same number of terms in this new expression and in (\ref{kthterm}). Comparing
the coefficient
of $a_{i + lp^n}(n+1)$ shows that we  need to compare
$f(i)$ and $f(i + lp^n)$.  Since $f$ was uniformly continuous 
($R$ is compact) these two function
values are uniformly bounded, thus the difference of the sums (\ref{kthterm}) 
and (\ref{kplus1term}) is small uniformly  and so the limit must exist.


To finish we need to show that this limit computes
the value of the linear functional $\mu$.  To this end, we investigate
the map, call it $\rho$, that sends the sequence 
\[
\left\{ s_n = \sum\limits_{i = 0}^{p^n -1} a_i(n) (t+1)^i \right\}
\]
associated to the measure $\mu$ to the linear functional $\mu'$ defined
by
\[
\int\limits_{G} f(t) d\mu' = 
\lim\limits_{k \rightarrow \infty} \sum\limits_{i = 0}^{p^n -1} a_i(n) f(i).
\]
We will show that $\rho$ is the identity.



The function $\rho$ is a $S$ module map and is the identity 
on the Dirac measures.
The restriction of  $\rho$ to $\Z/p^n$ is also the identity since everything can be computed
via the Dirac measures there.  Thus the diagram
\[
\xymatrix{
M(\Z/p^{n+1})  \ar[r]^{\rho} \ar[d] &   M(\Z/p^{n+1}) \ar[d] \\
M(\Z/p^n)  \ar[r]^{\rho}   &   M(\Z/p^n)
}
\]
commutes and has identity morphisms for horizontal arrows.  Taking limits we see that  
\[
\rho \colon M(G, S) \to M(G, S)
\]
must be  the identity.
\end{proof}

\subsection{A Special Case  of Katz's Theorem}
In this
section we will verify Theorem \ref{katzthm} in the case $G = \G_m$.
To help the flow of the proof we check the result  for the Dirac measures.
Let $ G = \G_m$ and $t$ the coordinate given by $u-1$.  The power series associated to the 
Dirac measure $ \psi_l$  in this case is 
$(1 + t)^l$.

Suppose $l$ 
is prime to $p$, then $\psi_l$ is clearly supported on the units.
We will write $\zeta$ for a primitive  $p$ root of unity over $\Z_p$ and $\zeta_j = \zeta^j - 1$.
We compute:
\begin{eqnarray*}
\sum\limits_{j=0}^{p-1} (1 + (t +_G \zeta_j))^l = 
\sum\limits_{j=0}^{p-1} (1 + t + \zeta^j -1 + t\zeta^j - t )^j \\
= \sum\limits_{j=0}^{p-1} (\zeta^j (t + 1))^l \\
= (t + 1)^l\sum\limits_{j=0}^{p-1} \zeta^{jl} = 0.
\end{eqnarray*}
The last sum vanishes because  $p$ is prime to $l$ and we are summing
over the $p$-th roots of unity.

\begin{lemma}[N. Katz] \label{easykatz} 
Suppose $t$ is the coordinate on $\G_m$ given by $u -1 $ with associated formal group law $F$. 
A measure $\mu$ is supported on the units  if and only if 
the corresponding power series $f_{\mu} (t) \in S[[t]]$ 
satisfies 
\[
\sum\limits_{j = 0}^{p-1}  f_{\mu}( x +_F \zeta_j) = 0 
\]
\end{lemma}


\begin{proof}

First, suppose we have a measure $\mu$ that is supported on the units  and the corresponding
power series
\[
f_\mu (t) = \sum\limits_i  \left( \int\limits_{G} { x \choose i } d \mu (x) \right) t^i
\]
given by Mahler's theorem.

The coefficient of $t^l$ in the trace of $f_\mu(t)$ is 
\[
\int\limits_{G} \sum\limits_{j = 0 }^{p-1} \zeta^{jx } { x \choose l} d \mu (x).
\]
However, the integrand  vanishes on the units and so by hypothesis
the integral must be zero.  


For the converse, suppose the trace of $f_\mu (t)$ is zero.
This time we know
\begin{eqnarray} \label{x_choose_vanish}
\int\limits_{G} \sum\limits_{j = 0 }^{p-1} \zeta^{jx } { x \choose l} d \mu (x) = 0
\end{eqnarray}
since it computes the coefficient of $t^l$ in the trace of $f_\mu (t)$.  
Using the Riemann sum representation (Proposition \ref{intform} ) we can compute (\ref{x_choose_vanish})
via
\[
\lim\limits_{n} \sum\limits_{i = 0}^{p^n - 1} a_i (n)
\sum\limits_{j = 0 }^{p-1} \zeta^{j i } { i \choose l} = 0.
\]
We can rewrite this expression as
\begin{eqnarray} \label{p_divides}
\lim\limits_{n} \sum\limits_{\substack {i = 0 \\  (i, p) \neq 1} }    ^{p^n - 1} a_i (n) { i \choose l} = 0.
\end{eqnarray}
We have dropped the factor of $p$ since $R$ is torsion free.

Now suppose $f \colon T_p G^\vee \to \Z_p$ is a continuous that vanishes on $T_p G^\vee$.
By Mahler's theorem we can represent $f$ as
\[
f(x) = \sum\limits_{l} b_l { x \choose l }.
\]
To compute its integral we need to calculate

\begin{eqnarray*} \label{vanish}
\lim\limits_{n} \sum\limits_{i = 0}^{p^n -1 } a_i (n) f (i) =  \\
\lim\limits_{n} \sum\limits_{i = 0}^{p^n -1 } a_i (n)  \sum\limits_{l = 0} b_l { i \choose l} = \\ 
\sum\limits_{l = 0}   b_l \lim\limits_{n} \sum\limits_{i = 0}^{p^n -1 } a_i (n)   { i \choose l} =  \\
\sum\limits_{l = 0}   b_l \lim\limits_{n}
\sum\limits_{   \substack{  i = 0 \\ (i, p ) \neq 1} }^{p^n -1 } a_i (n)   { i \choose l} =  0.
\end{eqnarray*}
The second equality follows since $T_p G^\vee $ is compact; by uniform continuity
\[
\int\limits_{G} \sum\limits_{n} b_n { x \choose n} d \mu =
\sum\limits_{n} \int\limits_{G} b_n { x \choose n} d \mu.
\]
The third equality follows from the hypothesis $f ( \left( T_pG^\vee \right)^\times) = 0$.
The last  equality follows from the calculation in  (\ref{p_divides}).
\end{proof}

\subsection{Proof of the General Case}
Let $\phi$ be a lift of the Frobenius over $\F_p$ to $\Z_p$.  The power series $\phi$ has coefficients in $\Z_p$
and satisfies 
\begin{eqnarray*}
\phi = p t + O(t)  \\
\phi \equiv t^p \mod p.
\end{eqnarray*}
Let $G = G_\phi$ be the associated Lubin-Tate
formal group law over $\Z_p$ with special endomorphism $\phi$.
Since $\G_m$ is the universal formal group law of height one 
there exists a strict isomorphism  
\[
\Theta \colon \G_m \to G.
\]

\begin{theorem} \label{generalkatz}
Suppose $t$ is a coordinate on $G$.
Under the identification
\[
\mathcal{D}_G \colon M(G, S) \to S[[t]]
\]
a measure $\mu$ is supported on the units if and only if 
\[
T_G f_\mu (t) = 0.
\]
\end{theorem}

\begin{proof}

The diagram
\[
\xymatrix{
M(\G_m ^\times  ,  S) \ar[r]   \ar[d]^\Theta         &         M( \G_m, S ) \ar[d]^\Theta      \\
M(G^\times, S) \ar[r]  & M(G, S)
} 
\]
commutes and the vertical arrows are isomorphism.

Suppose the trace of  $f(t) \in S[[t]]$ with respect to $G$ vanishes.
Under the change of variables $t \mapsto \Theta (s)$, $T_G f(t) = 0 $ 
maps to $T_{\G_m} f (\Theta(s)) = 0$.  Observe that
the set of $\Theta(v)$ for $v \in G[p]$ is the same as the set of $u \in  \G_m [p] $ since $\Theta$ is 
an isomorphism.

Pulling back along $\Theta$ produces the change of variables formula
\[
D_F ^k f(t) |_{t=0} = \Theta^* (D_F )^k f( \Theta (s)) |_{s = 0} = D^k f( \Theta(s) |_{s = 0}.
\]
It follows from Lemma \ref{d_compatible}
that the measure associated to $f(t)$ via the isomorphism 
\[
\mathcal{D}_G \colon M(G, S) \to S[[t]]
\]
coincides with the measure determined by $f ( \Theta (s))$ under
\[
\mathcal{D}_{\G_m} \colon M(G, S) \to S[[s]]
\]
as they produce the same moments.
The claim now follows from Lemma \ref{easykatz}.
One can use a similar argument to prove the ``only if'' part of the theorem.  
\end{proof}

\section{Measures and  $\CP^\infty$}

A measure on $\G_m^\times$ determines a cohomology operation on $K_p$, recall Theorem \ref{kpmeasure}
In this section we describe
the effect of these operations on $\CP^\infty$.  
\begin{theorem} \label{effectcp}
The effect of an operation $\mu$ on the the tautological line bundle over 
$\CP^\infty$ is the series $f_\mu$.
\end{theorem}
Once we understand the proof of this fact, the equivalence of Theorem \ref{katzthm} and the result of Madsen, Snaith
and Tornehave will be an easy consequence.
The argument is based on  section 2.3 of \cite{MR0494076}

\subsection{Collection of Statements}

Recall that for $p$-adic K-Theory, the cohomology of $\CP^\infty$
can be computed via the cohomology of the cyclic groups of $p$-powers. 
Explicitly we have
\begin{lemma}
There is an isomorphism 
\begin{eqnarray} \label{bzp}
K_p^0 \CP^\infty \cong \lim\limits_{\leftarrow} K_p^0  \left( B \Z /p^i \right) .
\end{eqnarray}
\end{lemma}
We produced a  similar statement for $S$-valued measures on $G$
\begin{eqnarray} \label{mzp}
M(G, S)  \cong \lim\limits_{\leftarrow} M(\Gamma_n  , S ).
\end{eqnarray}

\begin{definition} \label{ahat}
Let $\hat {A}_p $ be the set of H-maps  from $\Omega^\infty K_p$ to itself.
\end{definition}
We quote a lemma of Adams describing $\hat{A}_p$.  It is Lemma $6$ in \cite{MR0251716}
\begin{lemma} [Adams] \label{adamslemma}  There is an isomorphism
\[
\epsilon \colon \hat{A} _p \cong K_p \CP^\infty
\]
which sends an operation $\mu$  to its effect on the tautological line bundle $\tline$.
\end{lemma}

The above lemma together with line (\ref{bzp}) allows us to write any element of $\hat {A}_p$ uniquely
as a sequence 
\[
\left\{ s_n = \sum\limits_{i = 0 }^{ p^n -1} a_i (n) \psi^i \right\}
\]
where $\psi^i$ is the Adams operation at $i$ and $a_i(n) \in \Z_p$.


Our goal is the following
\begin{theorem} \label{MST+KATZ}
Suppose we have chosen a coordinate for $\G_m$ with 
induced formal group law $F$.
Under these conditions, the diagram
\begin{eqnarray}
\xymatrix{
M(\G_m ^\times , \Z_p ) \ar[r]^{\mathrm{res}} \ar[d] &     M(\G_m , \Z_p) \ar[r]^-{\mathcal{D}_{\G_m}}  &  \Z_p [[t]] \\
K^0_p K_p \ar[r]^{\Omega^\infty} &  \hat{A}_p   \ar[r]^\epsilon &   K_p^0 \CP^\infty \ar[u] 
}
\end{eqnarray}
commutes.
\end{theorem}

Observe that an unstable operation  $\phi \in \widehat{A}_p$ is in the image 
of $\Omega^\infty$  if and only if the the corresponding measure is in the image
of the restriction map.




The proof of Theorem \ref{MST+KATZ} will take  a few steps.  

\subsection{The K-Theory of $\CP^\infty$}


We exploit the similarity of Line (\ref{bzp}) and Line (\ref{mzp}) to calculate the effect
of a measure on the cohomology of $\CP^\infty$.

The direct system
\[
.... \Z / p^{n - 1} \xrightarrow{ \cdot p } \Z / p^n \to ...
\]
leads to the inverse system
\[
....K  B ( \Z / p^n )  \to   K  B ( \Z / p^{n - 1}  )\to ...
\]
It is easier to describe the maps in the system if we use Atiyah's calculation to replace these gadgets with
the representation rings 
\[
... R( \Z / p^n ) \otimes \Z_p \to R( \Z / p^{n - 1} ) \otimes \Z_p  \to ...
\]
These are easily computed as 
\[
... \Z_p[ s] / (s^{p^n} - 1)  \to \Z_p[s] / (s^{p^{n-1} } - 1) \to ...
\]
Where $s$ is the one dimensional irreducible representation
\[
\Z/p^n \to GL_1 ( \C) \cong \C^\times
\]
that sends the generator $\gamma$ to the appropriate primitive $p$-th root of unity.  These one dimensional 
representations are compatible under the structure  maps. Thus 
we have no notation to describe the source of the representation.

We can now describe the maps involved.  It is very similar to what we saw in the measure theory situation involving
$\Gamma_n$.
The map induced on the representation rings by multiplication by $p$ sends the generator $s$ to $s^p$.
Thus we have for a general element
\[
\sum\limits_{i = 0}^{p^n - 1} a_i (n)  s^i \to \sum\limits_{i = 0}^{p^n - 1} a_i (n) s^{ip}.
\]
Collecting terms once more, we can rewrite the image as
\[
\sum\limits_{j = 0 }^{p^{n-1} - 1}    \sum\limits_{l = 0}^{n-1} a_{j + l p^{n -1}} s^{jp}.
\]

We can connect this to unstable operations via an
isomorphism
\[
\hat {A_p} \cong \lim\limits_n K_p^0 B ( \Z /p^n ) \cong \lim\limits_{k} R(B \Z/p^n ) \otimes \Z_p 
\]
such that any operation corresponds to a sequence
\[
\left\{ s_n = \sum\limits_{i = 0}^{p^n - 1} a_i (n) s ^i \right\}
\]
where $\psi^i$ matches with  $s^i$
inside 
\[
K_p^0 B( \Z/ p^n ) \cong \Z_p[[s]] / ( s^{p^n} - 1 ).
\]
This is the approach of Madsen, Snaith and Tornehave.

\begin{lemma}
The diagram
\[
\xymatrix{
M(\Z/p^n ) \ar[r]  \ar[d]    &             M(\Z/p^{k - 1} ) \ar[d] \\
\Z_p [t] /  (s^{p^k } -1)  \ar[r]    &      \Z_p [t] / (s^{p^{k - 1}} - 1)
}
\]
where the vertical maps send $d \psi_i$ to $s^i$ commutes.
\end{lemma}

\begin{proof}
Obvious.
\end{proof}


Indeed, the vertical arrows are isomorphisms and so the limits are isomorphic:
\[
\lim\limits_{k} M ( \Z / p^k , \Z_p) \cong \lim\limits_{k} \Z_p[s] /  (s^{p^k } - 1).
\]
We are now in the situation to compute over $\Z_p$. 
Suppose $\mu \in M(\Z_p^\times)$.   Let $\bar{\mu}$ be the corresponding element of 
$M(\Z /p^k )$.  It is represented
by a formal sum
\[
\sum\limits_{i = 0}^{p^n - 1} a_i (n) d \psi_i.
\]

Likewise, if we consider $\mu$ as an element of $\hat{A_p}$ and then consider its image inside 
$K_p ( B\Z/p^n)$ it is described
by 
\[
\mu (s) = \sum\limits_{i = 0 }^{p^n -1 } b_i (n) s^i.
\]
We have described the map between these two gadgets.  It sends $s^i$ to $d \psi_i$.  
It remains to check if
\[
a_i (n) = b_i (n)
\]
for each $i$ if we consider $\mu$ as a measure on $\Z_p^\times$ or as an operation.

It's clear that if $\mu$ is a sum of Adams operations, these two coefficients agree since they are both counting
the  number of copies of an Adams operation in the original operation.  This remains true when
passing to limits and the result of the discussion is
\begin{proposition} \label{measure+ktheoy}
The diagram 
\[
\xymatrix{
M(\Z_p^\times) \ar[d] \ar[r]       &       \lim\limits_{k} M(\Z/p^n, \Z_p)  \ar[d] \\
K_p^0 K_p \ar[r]   &      \lim\limits_{k} K^0_p (B \Z/ p^k)
}
\]
commutes.
\end{proposition}

\subsection{Proof of Theorem \ref{MST+KATZ}}

In order to prove Theorem \ref{MST+KATZ} we will first verify it for a specific coordinate on $\G_m$.  
We will then use formal techniques to produce the result
for any other coordinate.

We are studying the diagram
\begin{eqnarray}
\xymatrix{
M(\Z_p ^\times ) \ar[r] \ar[d] & \lim\limits_k M( \Z/p^n ) \ar[d] \ar[r]
 & M(\Z_p)  \ar[r]  
&     \Z_p[[t]] \\
K_p^0 K_p  \ar[r]    &   \lim\limits_k  K^0_p (B/\Z/p^n)  \ar[r]   & \hat{A_p}  \ar[r]  & K (\CP^\infty) \ar[u].
}
\end{eqnarray}
Proposition \ref{measure+ktheoy}
shows that the left-hand rectangle commutes.  Two complete our proof, we have to compute the
the two paths around the right-hand rectangle.

The top horizontal arrow is computed via 
\begin{eqnarray} \label{line-int}
\mu \mapsto \sum\limits_{j = 0 }^\infty  \left(  \int\limits_{\Z_p} { x \choose j } d \mu \right) t^j 
\end{eqnarray}
In other words we are computing $f_\mu (t)$, the power series associated to the measure by Mahler's theorem.  
We can use the Riemann sums description
of integration in Proposition \ref{intform} to compute the
right hand side of (\ref{line-int}).  If $\mu$ is represented by the sequence
\[
\left\{ s_n = \sum\limits_{i} a_i(n) d\psi_i \right\}
\]
then the $j$-th coefficient of (\ref{line-int}) can be computed via
\[
\lim\limits_n \sum\limits_{i = 0 }^{p^n - 1} a_i (n) { i \choose j}.
\]

The bottom row is similar.  The isomorphisms between $\hat {A_p}$ and $K^0_p ( \CP^\infty)$ sends an operation
to its effect on the tautological line bundle.
The element $\mu ( \tline)$ inside the K-Theory of $\Z /p^n$ is represented by the element
\[
\mu (s) = \sum\limits_{i=0}^{p^n - 1} a_i (n) s^i.
\]
with the same  $a_i (n)$ by Proposition \ref{measure+ktheoy}.
The sequence
\[
\left\{ s_n = \sum\limits_{i = 0}^{p^n - 1} a_i (n) s^i \right\}
\]
is mapped to the sequence 
\[
\left\{ s_n = \sum\limits_{i = 0}^{p^n - 1} a_i (n) \psi^i \right\}
\]
inside $\hat{A}_p$.  The isomorphism of Adams sends this sequence to 
\[
\left\{ s_n = \sum\limits_{i = 0}^{p^n - 1} a_i (n) \tline^i \right\}
\]
inside $K_p^0 \CP^\infty$.
Finally, under the isomorphism of this cohomology group with a power series ring
the sequence is sent to 
\[
\left\{ s_n = \sum\limits_{i = 0}^{p^n - 1} a_i (n) (1 + t)^i \right\}.
\]
We can compute the power series representing this sequence in a natural way using the fact that
$\Z_p[[t]]$ is complete with respect to the $(p, t)$ topology.

Putting all of this together, the bottom row of our diagram sends $\mu$ to 
\[
\lim\limits_{k} \sum\limits_{i = 0}^{p^n - 1} a_i( n) \tline^i  =  \lim\limits_{k} \sum\limits_{i = 0}^{p^n - 1} a_i( n)(1 + t)^i =
\lim\limits_{k} \sum\limits_{i = 0}^{p^n - 1} a_i( n) \sum\limits_{l = 0}^i { i \choose l } t^l
\] 
Pulling off the coefficient of $t^j$ and comparing with the calculation of 
(\ref{line-int}) we obtain the same element in the power series ring as claimed.

\subsection{Changing Coordinates}

We have  asked for a specific coordinate to be chosen
in order to make the calculation of the previous section.
To finish the theorem we have to understand how
the theory changes as we change the coordinate. 

Suppose for the moment that we have another coordinate on the formal multiplicative
group and thus a new group law for K-Theory called $F$.

Let$ \Theta \colon  F \to (x + y + xy) $
be a strict isomorphism whose existence follows from deformation theory.
We write $s$ for the coordinate on $\G_m$ and $t$ for that of $F$.  
We have the change of variables $t \mapsto \Theta(s)$.
Recall the invariant derivation $D = (1 + s) \frac{d}{ds}$ and
\[
D_F = \frac{ 1} {\logf (t) } \frac{d}{dt} =  F_2 (t,0 ) \frac{d}{dt}.
\]

We can now state our main result for this section:
\begin{proposition} \label{measure_effect}
Under the bijection in Proposition \ref{diffop}
the effect of $\mu$ on $\tline$ is
given by $f_\mu (t)$.
\end{proposition}

\begin{proof}
The point is that pulling back $D$ along $\Theta$ induces $D_F$ and this is compatible with our identification
of $M(G, S)$ with $S[[t]]$, recall Lemma \ref{d_compatible}.
Write
\[
\Theta \colon \Z_p[[t]] \to \Z_p[[s]]
\]
for the map that sends $t$ to $\Theta (s)$. 
It is classical that the diagram
\[
\xymatrix{
[K_p, K_p]  \ar[r]^F \ar[dr]_G & \Z_p[[t]] \ar[d]^\Theta  \\
&         \Z_p[[s]] 
}
\]
commutes where the formal group law labels the isomorphism.  
If follows that if  $ \tline = (1 + s) \in \Z_p[[s]]$
then $\tline = 1 + \Theta^{-1}(t) \in \Z_p[[t]]$.

Pushing symbols around produces
\[
\int\limits_{G} \left( 1 + \Theta^{-1} (t) \right)^x d \mu (x) =   f_\mu (t)
\]
and the claim is verified.
\end{proof}





\section{Bernoulli Numbers}



Let $R$ be a  complete local ring that is torsion free with perfect residue field of characteristic $p$. 
Consider the collection of triples $(G, S, t)$ where $S$ is a torsion free complete local algebra over
$R$, $G$ is a finite 
height formal group over $S$  and $t$  is a coordinate on $G$.

\begin{definition} \label{berndef} A Bernoulli number of weight $k \geq 0$  over $R$ is a function $B$ that sends
a triple  $(G,  S, t)$ to an element of $S  \otimes \Q$ such that

\begin{enumerate}
\item  $B$ is invariant under isomorphic data

\item For any $a \in S^\times$ we have $B(G,S,  a t) = a^k B( G , S,  t)$. 

\item $B$ behaves well under pull backs.

\end{enumerate}

\end{definition}

The coordinate $a t$ is the composition
\[
G \xrightarrow{t} \affline \xrightarrow{a} \affline.
\]
The third condition requires some explanation.
Suppose $G$ lives over $S$ and we have a map
\[
g \colon  S \to S'
\]
of complete local $R$ algebras.
The pull back $g^* G $ is a formal group over $S'$.  
Condition three says there must be an equation in $S'$ of the form
\[
B \left( g^*G, S', g^* t \right)  = g\left(  B\left( G, S, t \right) \right).
\]
In other words, the Bernoulli numbers of a pull back are the push forward of the Bernoulli numbers.
We will often simplify the notation $B(G, R, t)$ to $B(G, t)$ in the case $S = R$.
Definition  \ref{berndef} is formal version of a Katz $p$-adic modular form \cite{MR0447119}

\subsection{The Fundamental Example}

In this section we will construct the fundamental example of a Bernoulli number of weight $k$.  It is unclear
at the moment whether or not we construct the unique Bernoulli number of weight $k$ up to multiplication by a constant.
In addition, we are unable to produce
a coordinate free version of the fundamental example.  The canonical derivation
\[
D : \strg \to Lie^\vee \otimes \strg
\]
may permit the removal of the coordinate from our test data.  However, it is not clear 
what the target for the Bernoulli numbers should be in that situation.

Recall that any formal group over a torsion free ring is isomorphic to the additive group 
after tensoring with the rationals.
Let $\logg$ be a formal group logarithm for $G$.  It is an isomorphism 
\[
\logg \colon  G \to \G_a
\]
of formal groups over $S \otimes \Q$
between the additive group and $G$.

Let $u$ be the coordinate on the formal additive group
\[
\G_a \xrightarrow{u} \affline
\]
where $G_a = \spf  \Z_p[u]  $.

The function  $\logg t $ corresponds to the appropriate composition in the diagram
\[
\xymatrix{
G \ar[d]_\logg \ar[r]^t &  \affline \\
\G_a \ar[r]^u &   \affline.
}
\]
It is an element of $O_G \otimes \Q$
and we can consider the element
\[
\frac{\logg t }{t} \in O_G ^\times \otimes \Q.
\]

Denote the maximal ideal of $\strg$ by $m_{\strg}$.  
The Iwasawa  logarithm is a homomorphism 
\[
\log \colon 1 + m_{\strg} \to \strg \otimes \Q
\]
of abelian groups.
It is constructed by considering the power series expansion of $\log (1 -x )$ about zero 
and extending to a $p$-adically complete ring.

Recall  the invariant derivation 
\[
D_F = F_2 (t, 0) \frac{d}{dt}
\]
and the associated map 
\[
\beta^k \colon  \strg \to S
\]
in Definition \ref{betak}.

Our fundamental example of a Bernoulli number is 
\begin{definition} \label{bernk}
Define
\[
B^k (G, S, t) := \beta^k \left( \log \frac{\logg t}{t} \right)
\]
\end{definition}
It is straightforward to verify the three properties of Definition \ref{berndef}.

\begin{example} \label{tood_exam}
Consider the triple $(\Z_p, \G_m, t )$ where $t$ is the coordinate $1 -u $.
In this situation, the function $B^k$ returns the classical Bernoulli numbers $- \frac{B_k}{k}$.
One way to see this is to use the following approach.

Consider the expansion
\[
\frac{t \expg ' t}{\expg t} = \sum\limits_{k \geq 0 }  b_k \frac{t^k}{k!}.
\]
Using logarithmic differentiation, we can show that 
\[
\frac{t}{\expg t } = \exp \left(  \sum\limits_{k \geq 1} -\frac{b_k}{k} \frac{t^k}{k!} \right).
\]
To compute $\beta^k \log \frac{ \logg t}{t}$, we pull back along the homomorphism $\expg t $ and obtain
\[
\frac{d^k}{dx^k} \log \frac{  \logg ( \expg x )}{ \expg x } |_{x = 0 }.
\]
This can be simplified to 
\[
\frac{d^k}{dx^k} \log \frac{ x }{ \expg x } |_{x = 0 }
\]
and we can read off the value as $-\frac{b_k}{k}$.

In this example, $\expg x = 1 - e^{-x}$ so
\[
\frac{ x  \expg ' x }{\expg x} = \frac{ x e^{-x}} {1 - e^{-x}} = \frac{x}{e^x - 1} = \sum\limits_{k \geq 0} B_k \frac{x^k}{k!}
\]
and we see that $B^k (\Z_p, \G_m , t ) = - \frac{B_k}{k}$ as claimed.
\end{example}


The construction of $B^k$ is natural with respect to change of coordinates.
For example, if $s$ is another coordinate than the diagram 
\[
\xymatrix{
O_G^\times \otimes \Q  \ar[d]  \ar[r]^\log &      O_G   \otimes \Q  \ar[d]     \ar[r]^{\beta^k}        &  S \otimes  \Q \ar[d] \\
O_G^\times  \otimes \Q  \ar[r]^\log              &  O_G  \otimes \Q   \ar[r]^{\beta^k} &  S \otimes \Q  
}
\]
communes where the vertical maps are induced by the change of coordinate $t \mapsto s$.

\section{The $V$ Operator}

The quotient of a formal group by a finite subgroup has played an important role in the study of elliptic curves
and more recently in algebraic topology.
We begin by recalling a fundamental result due to Lubin \cite{MR0209287}.

\begin{theorem}[Lubin] \label{lubin_thm}
If $F$ is a formal group of finite height over $R$ and $H$ is a finite subgroup scheme
of $F$ then $F/H$ is a formal group.  
\end{theorem}
A proof of this statement can be found in Lubin's original article, the work of Demazure and Gabriel or in 
\cite{MR2045503}.
The quotient formal group comes with an isogeny $f_H \colon  F \to F/H$ with kernel the group scheme $H$.
When the subgroup $H$ is invariant under the appropriate Galois action, both $f_H$ and the quotient $G/H$ can be defined
over $R$.

We are interested in the case when $H$ is the kernel of the isogeny induced by multiplication by $p$.
We  use the construction of the quotient to define an operator on Bernoulli numbers
analogous to that in the theory of Katz modular forms.  


If $t$ is a coordinate on $G$ then $f_H$ induces a map of formal group laws.  The associated power series
\[
f_H^t = \prod\limits_{c \in G[p] } t +_G c 
\]
is Lubin's isogoney.  
We denote by $t_H$ the coordinate induced on the quotient $G/H$.
It is determined by insisting that the diagram
\[
\xymatrix{
G \ar[r]^{f_H} \ar[d]_t &   G/H \ar[d]^{t_H} \\
\affline \ar[r]_{f^t_H} &  \affline 
}
\]
commute, see \cite{MR1344767}.

\begin{definition} \label{voperator}
If $B$ is a weight $k \geq 0$ generalized Bernoulli number
let  $B|V$ be the function on triples $(G, S, t)$ whose values are
\[
B|V (G, S, t)  := B ( G / H, S,  t_H)
\]
where $H$ is the kernel of multiplication by $p$ on $G$.
\end{definition}

\begin{lemma}The function $B|V$ is a Bernoulli number.
\end{lemma}

\begin{proof}
This is a formal exercise involving the naturalality  of the quotient.
\end{proof}

\subsection{The $K(1)$ Local Logarithm}

Recall that for an $E_\infty$ ring spectrum $R$, there is an associated ring spectrum $gl_1 R$ known
as the unit spectrum.  This spectrum was first considered in \cite{MR0494077}.  Some of the formal
properties of this spectrum, along with its construction can be found in \cite{MR2219307}.

\begin{theorem}[Bousfield, Kuhn] \label{bk_functor_theorem} 
There is a functor $\Phi$ from spaces to spectra such that
for any $E_\infty$ ring spectrum $R$, there is a weak equivalence
\[
\Phi \Omega^\infty (R) \sim \Loc R.
\]
\end{theorem}
The result as stated is due to Bousfield and for higher chromatic level follows from work of Kuhn 
\cite{MR901254}, \cite{MR1000381}.
Since $\Omega^\infty R$ and $GL_1 R$ have  weakly equivalent base point componets, Theorem \ref{bk_functor_theorem} 
produces a natural transformation of cohomology theories. 
\[
l_1 \colon  gl_1 R \to \Loc gl_1 R \cong \Loc R.
\]
Thus if $X$ is a space we obtain a  group homomorphism
$l_1 \colon \left( R^0 X \right)^\times \to (\Loc R)^0 (X)$
that is natural in $X$.
Rezk calculates the effect of this
map at all chromatic levels and for any commutative $K(n)$ local $S$-algebra $R$ \cite{MR2219307}. 
The spectrum  $K_p$ is a $K(1)$ local $E_\infty$ spectrum; the above discussion yields
\[
l_1 \colon  gl_1 K_p \to \Loc gl_1 K_p \cong \Loc K_p \cong K_p
\]
\begin{theorem}[Rezk] \label{rezk_thm}
Suppose $X$ is a finite complex and $\psi^p$ is the Adams operations at $p$ in $K_p$.
If
\[
x \in \left( K_p^0 ( X)  \right)^\times
\]
then 
\[
l_1 x = \left( 1 - \frac{\psi^p}{p} \right) \log x = \frac{1}{p} \log \frac{x^p}{\psi^p x}
\] 
\end{theorem}



We can use the fact that
\[
K_p \CP^\infty = \lim\limits_{\leftarrow} K_p B \Z/p^n
\]
to compute the effect of  $l_1$ on the cohomology of $\CP^\infty$.
\begin{proposition} \label{betal1}
Suppose $ a \in \Z_p^\times - \left\{\pm 1 \right\}$. 
If $\langle a \rangle (t) =  \frac{ [a]_G (t)}{t} $ then
\[
\beta^k l_1 \langle a \rangle (t) = ( 1- a^k) \left(  B^k (\G_m ,\Z_p,  t) -  \frac{B^k |V (\G_m, \Z_p,  t ) } {p}\right) 
\]
\end{proposition}

\begin{proof}
Observe that in Rezk's formula for $l_1$, the $V$ operator
plays the role of the Adams operation.  These are identical
since both are defined by the 
quotient with respect to the kernel of the multiplication by $p$ map on $G$.
On the other hand, 
the homomorphism $\expf (t)$ induces the formula 
$$
\beta^k f(t) = D^k_F  f(t)|_{t = 0} =  \frac{d^k}{dx^k} f \left(  \expf (x) \right) |_{x = 0}
$$
as $D_F$ pulls back to $\frac{d}{dx}$ along $\expf(t)$.
Using this formula and Definition \ref{bernk}, a quick calculation produces
the claim.
\end{proof}

\section{Coleman Norm Operator and the Ando Condition}

Our fundamental observation is the similarity between the Ando condition (Theorem \ref{andocondition}) 
and the
Coleman norm operator \cite{MR560409}.  

Let $H$ be the kernel of multiplication by $p$ on $G$.  It is a finite subgroup scheme 
since multiplication by $p$ is an isogeny.

\begin{theorem}[R. Coleman, Ando-Hopkins-Strickland, Demazure-Gabriel]
There is a multiplicative operator
\[
N_G \colon  O_G \to O_G
\]
such that if $t$ is any coordinate inducing the group law $F$ we have 
\[
N_G g \circ [p]_F (t) = \prod\limits_{c \in F[p^n]}  g(t +_F c ) 
\]
\end{theorem}

A number theoretic proof can be found R. Coleman's article.
See  \cite{MR2045503} for a version involving level structures on $G$.
As an example we have
\begin{lemma} \label{fixednorm}
If $a \in \Z_p$ then
the following conditions are equivalent:

\begin{enumerate}

\item $N_G t = t$

\item $N_G [a]_F (t) = [a]_F(t)$

\end{enumerate}

\end{lemma}

\begin{proof}
For ease of reading, we drop the subscript $G$ on the norm operator.
We have the chain of equalities.
\begin{eqnarray*}
Nt \circ [a]_F(t) \circ [p]_F (t) = Nt \circ [p]_F (t) \circ [a]_F (t) \\
= \prod\limits_{v \in F[p]} ( t +_F v ) \circ [a]_F(t) = \prod\limits_{v \in F[p]} [a]_F(t) +_F v \\
= \prod\limits_{v \in F[p]} [a]_F(t) +_F [a]_F(v) = \prod\limits_{v \in F[p] } [a]_F ( t +_F v) \\
= N [a]_F (t) \circ [p]_F (t).
\end{eqnarray*}
However, since $[p]_F(0) = 0$, it is an invertible power series under composition 
and we obtain the identity
$$
N t \circ [a]_F (t) = N [a]_F (t).
$$
We can now read off the claim.
\end{proof}

Recall from the discussion following Lubin's result (Theorem \ref{lubin_thm} )
that there is a quotient isogeny over $R$ associated to the finite subgroup $H$:
\[
f_H^* \colon  O_{G/H} \to \strg
\]

\begin{definition}
A coordinate satisfies the Ando condition if 
the two elements of $\strg$ determined by $f_H^*{ t_H}$ and $[p]_G  (t)$ coincide.
\end{definition}
Since Lubin's isogoney $f_H$ is given by $\prod\limits_{c \in H} t +_F c $, it is easy to show that the Ando condition is equivalent to the coordinate being fixed by the Coleman norm operator.

\begin{proposition} \label{andov}
If $t$ is a coordinate on $G$ that satisfying the Ando condition, then
\[
B^k (G, t) | V = p^k B^k (G, t)  \in R \otimes \Q
\]
\end{proposition}

\begin{proof}
We use a change of variables to compute $B_k (G, t ) |V$.
Let $x_H$ be the coordinate induced on the quotient and 
\[
f_H \colon  G \to G/H
\]
Lubin's isogoney.  By definition
\[
B^k (G,t ) |V = B^k ( G/ H , t_H).
\]
To compute the value we must pair $D_{G/H}$, the invariant derivation induced on the quotient, with
\[
\log \frac{ \log_{G/H} t_H}{t_H}.
\]
However, under base change along $f_H$ 
the derivation $D_{G/H}$ maps to $D_{G}$.
Thus
\[
\beta^k_{G/H} \log \frac{ \log_{G/H} t_H}{t_H} = \beta^k_G \log \frac{\log_{G/H} f_H t  }{f_H t }.
\]
There are two observations to make at this point.  First, the diagram
\[
\xymatrix{
G \ar[d]_\logg  \ar[r]^{f_H} &   G/ H \ar[dl]^{\log_{G/H}} \\
\G_a 
}
\]
is commutative.  The second, and more important, is that the Ando condition implies
\[
f_H t_H = [p]_G t
\]
as elements of $\strg$.
Since the $V$ operator is induced by the cokernel of multiplication by $p$ we
have
\[
B^k (G/H, t_H ) = \beta^k_G \log \frac{ \logg [p]_G (t) }{[p]_G t}. 
\]
A change of variables argument similar to the one used in the proof of 
Proposition \ref{betal1}
shows
\[
\beta^k_G \log \frac{ \logg [p]_G (t) }{[p]_G t}= p^k B^k (G, t).
\]
\end{proof}


Since $ \langle a \rangle  (t)$ is an element of $\left( K_p \CP^\infty \right)^\times$, 
$l_1 \langle a \rangle (t) $  is a power series with integral coefficients and so  determines
a measure on $\G_m$.  We have already computed the moments of this measure in Proposition \ref{betal1}. 
However, we are interested in producing moments on the $p$-adic units i.e.  
is $l_1 \langle a \rangle (t)$ in the image of the restriction map
\[
\mathrm{res} : M (\Z_p^\times, \Z_p ) \to M ( \Z_p, \Z_p)? 
\]
Recall the key result of Katz, Theorem \ref{generalkatz}.

\begin{proposition} \label{trace}
Suppose $G$ is a height one formal group over $\Z_p$ and $t$ is a coordinate that 
satisfies the Ando condition. 
For any $a \in \Z_p^\times - \left\{ \pm 1 \right\}$
\[
T_G l_1 \langle a \rangle_F (t) =
\sum\limits_{c \in F[p]} l_1 \langle a \rangle_F (t+_F c ) =  0.
\]
\end{proposition}

\begin{proof}
It follows from our discussion around  Lemma \ref{fixednorm} 
that if the coordinate $t$ satisfies the Ando condition then both 
$t$ and $[a]_F (t)$ are
fixed by the norm operator $N_G$.
Since $N_G$ is a multiplicative operator, the divided $a$ series 
\[
\langle a \rangle_F (t) = \frac{    [a]_F (t)  }{ t}  
\]
is also fixed by the norm operator.
We can now  rewrite the sum as a product using the formula of Rezk (Theorem \ref{rezk_thm}).  
Simplying the expression inside the finite product
will produce the claim. 
\end{proof}
This proof is formal.  If $f (t) \in \left( K_p \CP^\infty \right)^\times $ is any invertible power
series, then the trace of $l_1 N_G f(t)$ vanishes.  Moreover, one can show that the a restriced form
of the Bousfield-Kuhn idempotent
\[
\phi \colon \widehat{A}_p \to [K_p, K_p] \xrightarrow{\Omega^\infty} [\Omega^\infty K_p, \Omega^\infty K_p]
\]
is given by the map $ id - \frac{1}{p} T_G$.



\begin{corollary} \label{andobk}
If $t$ satisfies the  Ando condition then the sequence 
\[
B^k (G, t )  (1 - a^k) \left( 1 - p^{k-1} \right)
\]
arises as the moments of a measure on $G^\times$.
\end{corollary}

\begin{proof}
It follows from Theorem  \ref{generalkatz} and Proposition \ref{trace} 
that the measure attached to $l_1 \langle a \rangle_G (t)$  is supported on the units.
The calculation in Proposition \ref{betal1} togeether with Proposition \ref{andov} produces the sequence above as the moments
of this measure.
\end{proof}

\section{Obstruction Theory of AHR}


Given any space  $B$ over $BO$, we  can produce an 
$E_\infty$ ring spectrum $MB$.  The space of $E_\infty$
orientations is the homotopy pull back in the square
\begin{eqnarray} \label{mgsquare}
\xymatrix{
E_\infty (MB, R) \ar[r]  \ar[d]   &   \mathrm{Spectra} ( gl_1 S / b , gl_1 R) \ar[d]  \\
\mathrm{Spectra} (S, R) \ar[r] & \mathrm{Spectra} (gl_1 S , gl_1 R).
}
\end{eqnarray}
This description is  due to Ando, Blumberg, Gepner, Hopkins and Rezk and is discussed in generality in \cite{BG}.
The results that are closer to what we require can be found in the first few sections of \cite{AHR}.

\subsection{The Calculation for $MU$} \label{section_mu}

If $X$ is an $E_\infty$ ring spectrum, we  write
\[
i \colon  S \to X
\]
for the unit.  We abuse notation and write
\[
i \colon  gl_1 S \to gl_1 X
\]
for the map induced on the associated unit spectra.

We will mimic the notation of \cite{MR1869850} for connected covers of the space $BU$.  
For example
\[
BU \langle 4 \rangle = BSU
\]
is the four connected cover of $BU$ and
\[
BU \langle 0 \rangle = BU \times \Z.
\]
There are  associated spectra
\begin{eqnarray*}
bu \langle 0 \rangle  & = bu \\
bu \langle 2 \rangle &\\
bu \langle 4 \rangle & = bsu
\end{eqnarray*}
whose zeroth spaces is $BU \langle 2k \rangle$.

Let 
\[
u\langle 2 \rangle = \Sigma^{-1} bu \langle 2 \rangle. 
\]
The stable $j$-homomorphism leads to a cofiber sequence
\[
u\langle 2 \rangle  \xrightarrow{j} gl_1 S \to gl_1 S / u\langle 2 \rangle .
\]
\begin{definition}[Ando, Blumberg, Gepner, Hopkins, Rezk] \label{msu_def}
Define $MU$  as the homotopy pushout in the diagram of $E_\infty$ ring spectra 
\[
\xymatrix{
\Sigma^\infty_+ \Omega^\infty gl_1 S    \ar[d] \ar[r]  &    S \ar@{-->}[d] \\ 
\Sigma^\infty_+ \Omega^\infty gl_1 S / u \langle 2 \rangle   \ar@{-->}[r] &   MU 
}
\]
\end{definition}
Applying the functor $E_\infty( - , R)$ to this digram yields
\begin{eqnarray} \label{pullbackdiagram}
\xymatrix{
E_\infty (MU, R)   \ar[r]  \ar[d]           &  E_\infty (\Sigma^\infty_+ \Omega^\infty gl_1 S / u  \langle 2 \rangle  ,  R)\ar[d]    \\
\{i\}         \ar[r]   &      E_\infty ( \Sigma^\infty_+ \Omega^\infty gl_1 S , R). 
}
\end{eqnarray}
It is shown in \cite{BG} that 
$(\Sigma^\infty_+ \Omega^\infty, gl_1 )$ form an adjoint pair up to homotopy and the 
upper horizontal arrow in Definition \ref{msu_def} is the counit of  the adjuction. 
Applying the adjoint twice,  diagram \ref{pullbackdiagram} becomes 
\[
\xymatrix{
E_\infty (MU, R) \ar[r]    \ar[d]  &      \mathrm{Spectra} (gl_1 S /u \langle 2 \rangle , gl_1 R ) \ar[d]  \\
\{i\}  \ar[r]   &    \mathrm{Spectra} ( gl_1 S, gl_1 R)
}
\]
and the space $E_\infty (MU, R)$ is weakly equivalent to the homotopy pullback.
The long exact sequence in homotopy groups associated to this diagram 
shows that the components of the space of orientations are in bijection with null homotopies
of the sequence  
\[
\xymatrix{
u  \langle 2 \rangle \ar[r]  &   gl_1 S  \ar[r]^i & gl_1 R  \\
}
\]
In other words, we are interested in producing maps $\alpha \colon  gl_1 S /u \langle 2 \rangle \to gl_1 R$ 
in the homotopy category so that the diagram 
\[
\xymatrix{
u  \langle 2 \rangle \ar[r]  &   gl_1 S  \ar[r] \ar[d]^i   &     gl_1 S / u \langle 2 \rangle \ar@{-->}[dl]^\alpha  \\
& gl_1 R
}
\]
commutes.


The first step in the AHR approach to the description of $E_\infty$ orientations of complex K-Theory
is to show that there exists an $E_\infty$ map $MU \to K_p$.
\begin{lemma}The set $\pi_0 E_\infty(MU, K_p)$ is non empty. \label{munull}
\end{lemma}

\begin{proof}  
To see that there is such a null homotopy we will show that 
\[
[u \langle 2 \rangle , gl_1 K_p  ] = 0
\]
using the splitting of the units of K-theory due to Adams, Priddy and others.
Since $u \langle 2 \rangle $ is one connected, we can compute on the $1$ connected covers,
\[
[u \langle 2 \rangle , gl_1 K_p \langle 1 \rangle ]
\]
This follows from Proposition $5.19$ in \cite{AHR}.

\begin{proposition}[Adams-Priddy, Madsen-Snaith-Tornehave]
There is a splitting, 
\[
gl_1 K_p \langle 1 \rangle = \Sigma^2 H\Z_p \times \Sigma^4 bu^\wedge _p = \Sigma^2 H\Z_p \times bsu_p^\wedge . 
\]
\end{proposition}
The proposition  can be found on  page $416$ of Madsen, Snaith and Tornehave  \cite{MR0494076}.
Thus our calculation comes down to inspecting $[u\langle 2 \rangle, bsu^\wedge_p]$ and $[u \langle 2 \rangle, \Sigma^2 H \Z_p]$.
Both of these groups vanish as the K-Theory of $u$ and the singular cohomology of $u$ vanish.



These two observations combined with the splitting of $gl_1 K_p \langle 1 \rangle$
and a K\"unneth theorem shows that  
\[
[ u \langle 2 \rangle , gl_1 K_p \langle 1 \rangle ] \cong [ u \langle 2 \rangle , gl_1 K_p] = 0.
\]
Thus there is an $E_\infty$ map
$
MU \to K_p
$
and  
$
\pi_0 E_\infty (MU, K_p )
$
is non empty.
\end{proof}

\subsection{The Miller Invariant}

Haynes Miller showed that congruence condition in Theorem  \ref{ahrtheorem} is true for any orientation of complex
K-Theory \cite{MR686158}.  We review this result. 
Suppose $E$ is an even two periodic theory that is complex oriented such that $\pi_0 E = R$ is torsion free.
Recall the Bernoulli number $B^k$ from Definition \ref{bernk}.
If $\alpha \colon MU \to E$ is any ring map, write
\[
B^k(G, R, \alpha)
\]
for the effect of $B^k$ on the triple $(G, R, t)$ where $t$ is the coordinate induced by $\alpha$ and $G$
is the formal group of $E$.

\begin{theorem}[H. Miller] \label{miller}
If
\[
\alpha, \beta \colon MU \to E
\]
are two  naive ring maps of spectra, then
\[
B^k(G, R, \alpha) \equiv B^k(G,R, \beta) \mod R
\]
\end{theorem}

The congruence means that the difference of the two values is integral i.e. is an element of $R$.
Hence Miller's calculation provides a  necessary condition for $E_\infty$ orientability.  

\subsection{ The Characteristic map of AHR}

We require a definition from \cite{AHR} before stating the main result of this section.
Fix a positive integer $n$.  Consider the set $A_n$ of polynomials with $p$-adic rational coefficients that satisfy the following condition.
If $b \in \Z_p^\times$ and $h(z) = \sum\limits_{i \geq n}  a_iz^i \in A_n$ then
$
\sum\limits_{i \geq n }  a_i b^i 
$
is a $p$-adic integer.
Note that any polynomial in $A_n$ is a continuous function from $\Z_p^\times$ to $\Z_p$.

\begin{definition}[GKC] \label{gkc}
A sequence of $p$-adic integers $\left\{ t_k \right\}$ for $k \geq n$ satisfy the generalized Kummer congruences 
if for any $h(z) = \sum\limits_{i \geq n}  a_i z^i \in A_n$
we have
\[
\sum\limits_{i \geq n }  a_i t_i \in \Z_p.
\]
\end{definition}

\begin{example} \label{moments_gkc}
Suppose $G$ is a Lubin-Tate group of height one over $\Z_p$ with coordinate $t$.
If $\mu \in M(G^\times , S)$ then set
\[
 t_k = \int\limits_{G^\times} \langle D_F^k , - \rangle d \mu.
\]
If $h(z) = \sum\limits_{i \geq n}  a_i z^i \in A_n$ then 
\[
\int\limits_{G^\times} h(z) d \mu (z) = \sum\limits_{i \geq n}  a_i t_i \in \Z_p.
\]
Thus if $\left\{ t_k \right\}$ is the moments of a measure on the units, then 
it satisfies the generalized Kummer congruences.  The converse is also true; the set $A_1$ is dense in $\Cont (\Z_p^\times, \Z_p)$.

To obtain the classical Kummer congruences, observe that if
$m \equiv n \mod (p-1)p^{k-1}$
then by Fermat's little theorem
$$
\int_{G^\times} x^m d \mu \equiv \int_{G^\times} x^n d \mu \mod p^k.
$$
This approach to Kummer like congruences is attributed to B. Mazur.
\end{example}


\begin{lemma}
There is an isomorphism between 
\[
[bu \langle 2 \rangle , K_p \otimes \Q] 
\]
and the set of sequences $\left\{ b_k \right\}  \in \Q_p$ for $k \geq 1$.
\end{lemma}

\begin{proof}
This is Definition $5.9$ in \cite{AHR}.
Let $v^k$ be the Bott element in $\pi_{2k} bu$.  
The isomorphism sends a map $f$ to the sequence $s(f)$ defined by
\[
s(f)_k = f_* v^k.
\] 
That is we look at the effect of $f$ on $v^k$ in degree $2k$ for $k \geq 1$. 
\end{proof}

\begin{theorem}[AHR] \label{image_gkc}
Fix a positive integer $n$.
The  natural map
\[
K_p ^0 K_p \to [bu \langle 2 \rangle  , K_p \otimes \Q] \cong \prod\limits_{k \geq n } \Q_p 
\]
sending an operation $\mu$ to the sequence $\pi_{2k} \mu$
is injective.
Moreover, the image of
$K_p^0 K_p $ is identified with the set of sequences $\left\{ t_k \right\}$ satisfying the generalized Kummer
congruences for $k \geq n$.
\end{theorem}

\begin{definition}
Define 
\[
b \colon \pi_0 E_\infty(MU, K_p) \to [bu \langle 2 \rangle , K_p \otimes \Q] \cong \prod\limits_{k \geq 1} \Q_p.
\]
by 
\[
b ( \alpha ) = \left\{ B^k (\G_m ,\Z_p, \alpha) \right\}.
\]
\end{definition}
This is called the characteristic map in \cite{AHR}.

\begin{theorem}[AHR] \label{muorient}
For odd primes $p$, the characteristic map 
\[
b \colon \pi_0 E_\infty (MU, K_p) \to [bu \langle 2 \rangle, K_p \otimes \Q] 
\]
identifies the set $\pi_0 E_\infty (MU, K_p) $
with the set of sequences 
$b_k \in \prod\limits_{k \geq 1 } \Q_p $  satisfying

\begin{enumerate}

\item For any $a \in \Z_p^\times - \left\{ \pm \right\}$ the sequence 
$
b_k ( 1 - a^k ) ( 1 - p^{k-1} )
$
arises as the moments of a measure on $\Z_p^\times$ and

\item  $b_k \equiv - \frac{ B_k}{k} \mod \Z_p$.
\end{enumerate}

\end{theorem}

\begin{proof}
Since  $\pi_0 E_\infty (MU, K_p)$ is non empty, we can run the AHR program.  Their
argument for $MO \langle 8 \rangle $ orientations of $KO_p$ in \cite{AHR} can be used in this situation.
The details are in  \cite{bjw}.  We provide a sketch of the proof.

Recall that the complex unstable Adams conjectured was verified in \cite{MR549303}.
The conjecture implies that the the top horizontal row in the diagram
\[
\xymatrix{
J_c  \ar[r]     &     BU_{(p)}  \ar[r]^{1 - \psi^a}  &          BU_{(p)} \ar[d]^=   \\
GL_1 S _{(p)}  \ar[r]   &    \left(  GL_1 S / U  \right) _{(p)}  \ar[r]
&    BU_{(p)}   \ar[r]^J     & BGL_1 S_{(p)} \\
}
\]
is null as a map of infinite loop spaces for any $p$-adic unit $a$.  We can
fill this diagram in
\[
\xymatrix{
J_c  \ar[r]  \ar@{-->}[d]^{A}   &     BU_{(p)}  \ar[r]^{1 - \psi^a}
\ar@{-->}[d]^{B} &          BU_{(p)} \ar[d]^=   \\
GL_1 S _{(p)}  \ar[r]   &    \left(  GL_1 S / U  \right) _{(p)}  \ar[r]
&    BU_{(p)}   \ar[r]^J     & BGL_1 S_{(p)} \\
}
\]
using the universal property.
Applying the Bousfield-Kuhn functor $\Phi$ (Theorem \ref{bk_functor_theorem}) yields the top half of
\begin{eqnarray} \label{big_diagram}
\xymatrix{
\Loc j_c    \ar[r]  \ar[d]^{\Phi A}  & 
 K_p  \ar[d]^{\Phi B}  \ar[r]^{1 - \psi^a}   &   K_p \ar[d]^=  \\
\Loc gl_1 S \ar[d] \ar[r] &     \Loc gl_1  S  / u  \ar@{-->}[d]^\alpha   \ar[r]
&    K_p  \ar@{-->}[d]^\beta  \ar[r]  &        \Loc bgl_1 S  \ar[d] \\
\Sigma^{-1} \Loc gl_1 K_p \otimes \Q / \Z  \ar[r] \ar[d]^{l_1}      &   
\Loc gl_1 K_p  \ar[r] \ar[d]^{l_1}  &
\Loc gl_1 K_p \otimes \Q  \ar[r] \ar[d]^{l_1}
&   \Loc gl_1 K_p \otimes \Q / \Z  \ar[d]^{l_1} \\
\Sigma^{-1} K_p \otimes \Q /\Z  \ar[r]   &     K_p   \ar[r]^\tau
&   K_p \otimes \Q  \ar[r]  &      K_p \otimes \Q/\Z
}
\end{eqnarray}
We have used the fact
that $K_p$ is a model for $\Loc bu \langle 2 \rangle$.
The map $\Phi A$ is an expression
of the $K(1)$ local sphere at odd primes and is a weak equivalence
by Adams, Baird, Bousfield and Ravenell \cite{MR551009}, \cite{MR737778}.  
It follows that  $\Phi B$ is a weak equivalence.

A map $\alpha$ making Diagram (\ref{big_diagram}) commute
corresponds to a null homotopy of $u \langle 2 \rangle \to gl_1 S \to gl_1 K_p$.
Since $\Phi B$ and $l_1$ are weak equivalences, to give a map $\alpha$ up to
homotopy is equivalent to producing 
\[
l \circ \alpha \ \circ \Phi B \in [K_p, K_p].
\]
Let $t_k ( \alpha)$ be the sequence of moments attached to this operation.
Producing a map $\alpha$ is equivalent to giving  a map $\beta$ since $\Loc gl_1 S$ is torsion.
Write  $t_k (\beta)$ for the effect of $\beta \colon K_p \to \Loc gl_1 K_p \otimes \Q$ in homotopy. 
A priori, $t_k (\beta) \in \Q_p$.

The formula of Rezk (Theorem \ref{rezk_thm} ) for the $K(1)$ local logarithm shows that $l_1$ is
multiplication by $\left(1 - p^{k-1}\right)$ in degree $2k$.  Thus,
to give a a pair of maps  $\alpha$ and $\beta$ making Diagram (\ref{big_diagram})
commute is equivalent to producing a sequence $t_k (\beta)$ such that
\begin{eqnarray} \label{alpha_beta_sequence}
(1 - a^k) \left( 1 - p^{k-1} \right) t_k ( \beta) = \tau_* t_k (\alpha) = t_k(\alpha).
\end{eqnarray}
The sequence in Equation (\ref{alpha_beta_sequence})
satisfies the generalized Kummer congruences for $k \geq 1$ by
Theorem \ref{image_gkc}.
Example
\ref{moments_gkc} show $t_k (\alpha)$ must arises 
as the moments of a  measure on the units.

The congruence condition follows from Miller's theorem \ref{miller}.
\end{proof}

This is the first instance where we have insisted that $p$ be an odd prime.  The obstacle at the prime $2$
is the cofiber sequence involving the $K(1)$ local sphere.  At the even prime, we must use the Adams map
on \textit{real} K-Theory $KO_2$.  The upshot is that a global result for $MO\langle 8 \rangle $ i.e. $MString$
orientations is obtained for real K-Theory.  A method for classifying complex orientations for complex K-Theory
at the prime $2$ seems mysterious at the current juncture.

%

\section{$H_\infty$ Implies $E_\infty$ }

In this section we complete our proof of Theorem \ref{thm1}.
\begin{theorem}   \label{himpliese}
There exists an injective map  
\[
\xymatrix{
{\mathbb{D}} : H_\infty (MU, K_p) \ar[r]   &   \pi_0 E_\infty (MU, K_p). \\
}
\]
\end{theorem}

\begin{proof}
Consider the function  which sends the class represented by 
\[
\alpha \colon MU \to K_p
\]
satisfying the Ando condition to the sequence 
\begin{eqnarray} \label{charseq}
B^k (\alpha) = B^k (\G_m, \Z_p, \alpha).
\end{eqnarray}
Corollary \ref{andobk} shows that the sequence (\ref{charseq}) for $k \geq 1$ satisfies the first condition of the AHR classification 
(Theorem \ref{ahrtheorem}).  Miller's result (Theorem  \ref{miller}) shows that  the congruence condition is
satisfied.
To see that the map is injective, observe that distinct  classes, $\alpha$ and $\beta$ 
in 
\[
H_\infty (MU, K_p)
\]  
determine distinct sequences  $B^k (\alpha)$ and $B^k (\beta)$ 
and thus distinct components in $E_\infty (MU, K_p)$. 
\end{proof}

\begin{corollary}
There is an isomorphism
\[
\mathbb{D} \colon  H_\infty(MU, K_p) \to \pi_0 E_\infty (MU, K_p)
\]
\end{corollary}

\begin{proof}
We check that composition with the forgetful map 
\[
\mathbb{C} \colon \pi_0 E_\infty (MU, K_p) \to H_\infty(MU, K_p)
\]
is the identity.  To begin, suppose 
\[
\alpha \colon MU \to K_p
\]
is merely $H_\infty$.  The map $\mathbb{D}$ identifies $\alpha$ with  the $E_\infty$ orientation corresponding to the
sequence $B^k (\alpha)$.  Under the identification in Theorem \ref{muorient} 
this component is identified with
$\alpha$.  The other composition is similar.
\end{proof}
The author feels that this calculation locates the power operations in the classification of $E_\infty$ maps
due to AHR.  As an application we have

\begin{corollary}
For odd primes $p$, the Todd genus 
\[
MU \to K_p
\]
is an $E_\infty$ map.
\end{corollary}

\begin{proof}
The Todd genus corresponds to the orientation $1 - u$ on $\G_m$ over $\Z_p$.  The Bernoulli number $B^k$ 
returns $- \frac{B_k}{k}$ as seen in Example \ref{tood_exam}.
One can check that the Ando condition is satisfied for this choice of coordinate and so this the Todd
genus  is an $E_\infty$ map.
\end{proof}

The Todd genus corresponds to the measure on $\Z_p^\times$ whose moments 
are 
\[
- \frac{B_k}{k} \left( 1 - a^k \right) \left( 1 - p^{k-1} \right). 
\]
This is the measure considered in 
Chapter II section $6$ of \cite{MR0466081} and  is known as the Mazur measure.

\subsection{Why $\langle a \rangle_F (t) $?}

We close by producing a conceptual argument for the appearance of the divided $a$ series in our calculations.
First recall the cannibalistic class $\rho$.  Suppose
\[
\xymatrix{
V \ar[d] \\
X 
}
\]
is a complex vector bundle (of virtual rank zero) over a connected base space $X$.
It is classical that
the cohomology of the Thom space,
$K_p(X^V)$, is a  free rank one  module over the cohomology of the base.
A generator is known as a Thom class, $U(V)$, of the bundle $V$.
A complex orientation determines a Thom class and the 
Thom isomorphism
\[
\tau_V \colon K_p (X) \to K_p (X^V)
\]
is multiplication by this Thom class. 
The following deﬁnition is essentially due to Bott.

\begin{definition}
Suppose $a$ is a unit in $\Z_p$.
Given a complex bundle $V$ as above define a class
\[
\rho_a  U(V)  = \frac{ \psi^a U (V) }{U(V) } \in \left( K_p X \right) ^\times
\]
\end{definition}
A few remarks are necessary. First, since $\psi^a$ is a ring map, $\psi^a U (V) $ 
is a generator for $K_p (X^V )$ as
a $K_p (X )$ module. 
The ratio in this definition  is the unique element of $K_p (X)$ 
that takes the Thom class $U(V)$ to the
class represented by $\psi^a \left( U ( V) \right)$.

The cannibalistic class $\rho_a$ produces an element in the units of the base.  As the bundle $V$ was arbitrary
we have constructed a map of spectra
\[
bu \langle 2 \rangle  \to  gl_1 K_p.
\]
\begin{example}
Let $\tline$ be the tautological line bundle over $\CP^\infty$.
Suppose we have a complex orientation $\alpha$ of $K_p$.   Let $t$ be the
Euler class of $\tline$ induced by this orientation.
In this situation we have the formula
\[
\rho_a  (t) = \frac{    [a]_F (t)}{t} = \langle a \rangle _F (t).
\]
\end{example}

Using the obstruction theory described above, an $E_\infty$ orientation 
is given by a  null homotopy of
\[
u \langle 2 \rangle \to gl_1 S \to gl_1 K_p.
\]
We have shown that such a null homotopy exists in 
Section \ref{munull}.
If we precede the above chain of spectra
with the stable Adams conjecture and follow with the $K(1)$ local logarithm we obtain the diagram
\[
\xymatrix{
           &               u \langle 2 \rangle    \ar[d]^{\psi^a - 1} \\ 
       &  u \langle 2 \rangle  \ar[r] \ar[d] &  gl_1 S    \ar[r]  &       gl_1 K_p  \ar[d]_{l_1}\\ 
       &     C  \ar[d]              &                         &          K_p \\
\CP^\infty \ar[r]^{e (\tline)}   &       bu \langle 2 \rangle    \ar@{-->}[uurr] ^\rho \\
}
\]
The arrow labeled $e(\tline)$ is Euler class of $\tline$.

This diagram shows
that we  are comparing the given orientation with $\psi^a$ of that orientation. 
As seen from the diagram, this ratio is 
the cannibalistic class;
\[
\rho \colon bu \langle 2 \rangle \to gl_1 K_p.
\]
This class followed by $l_1$ 
is represented by  $l_1 \langle a \rangle_F (t)$ in 
$\left( K_p \CP^\infty \right)^\times$. 
We would like to contribute this observation to Charles Rezk.




\end{document}